\documentclass[10pt,reqno]{amsart}
\usepackage{floatflt}
\usepackage{xypic}
\usepackage{rotating}
\usepackage{bbm}
\usepackage{mathrsfs}
\usepackage{diagbox}
\usepackage{dutchcal}
\usepackage{extarrows}
\usepackage{cite}
\usepackage{amsfonts} 
\usepackage[dvipsnames,usenames]{color}
\usepackage{graphicx,latexsym,bm,amsmath,amssymb,verbatim,multicol,lscape}
\makeatletter
\@namedef{subjclassname@2020}{\textup{2020} Mathematics Subject Classification}
\makeatother
\newtheorem{theorem}{Theorem} [section]

\newtheorem{lemma}[theorem]{Lemma}

\numberwithin{equation}{section}

\def\bew{\begin{widetext}}
\def\eew{\end{widetext}}
\def\be{\begin{equation}}
\def\ee{\end{equation}}
\def\bea{\begin{eqnarray}}
\def\eea{\end{eqnarray}}

\allowdisplaybreaks
\allowdisplaybreaks
\begin{document}
\title[Another proofs of Zagier's and Murakami's formulas]
{Another proofs of Zagier's formula for multiple zeta values 
and Murakami's formula for multiple $t$-values}
\begin{abstract}
Let $l\ge 1$ be an integer. For any multiple index
$\mathbf{s}=(s_1,s_2,\cdots,s_l)\in\mathbb{Z}_{\geq 1}^l$
with $s_l>1$, the multiple zeta value
(MZV for short) is defined by
\begin{align*}
\zeta(s_1,s_2,\cdots,s_l):=\sum_{1\leq k_1<k_2<\cdots<k_l}
\frac{1}{k_1^{s_1}k_2^{s_2}\cdots k_l^{s_l}}
\end{align*}
and the multiple $t$-value is defined by
\begin{align*}
t(s_1,s_2,...,s_l):=\sum_{1\leq k_1<k_2<...<k_l}
\frac{1}{(2k_1-1)^{s_1}(2k_2-1)^{s_2}...(2k_l-1)^{s_l}},
\end{align*}
where if the index is empty, then we define the value $t(\emptyset):=1$.
We denote by $\{a_1,\cdots,a_k\}^d$ the sequence formed by repeating the
sequence $\{a_1,\cdots,a_k\}$ exactly $d$ times. Let
$H(r,s)=\zeta(\{2\}^r,3,\{2\}^s)$ and $T(r,s):=t(\{2\}^r,3,\{2\}^s)$.
Zagier's formula for the multiple zeta values $H(r,s)$ was an important
and key ingredient in the proof of Hoffman's conjecture. In this paper,
with the help of the Lei-Yu-Hong expressions for $H(r,s)$ and $T(r,s)$ as
well as Lupu's identity about rational zeta series involving Riemann
zeta values $\zeta(2n)$ and by establishing some identities about
binomial coefficients and a result about Kronecker symbol and arithmetic
functions, we present another proofs of Zagier's formula stating that for 
any nonnegative integers $r$ and $s$,
\begin{align*}
H(r,s)=2\sum_{k=1}^{r+s+1}(-1)^k\Big[\binom{2k}{2r+2}-\Big(1-\frac{1}{2^{2k}}\Big)
\binom{2k}{2s+1}\Big]\zeta(2k+1)\zeta(\{2\}^{r+s+1-k}),
\end{align*}
and Murakami's formula for the multiple $t$-values $T(r,s)$ asserting that
\begin{align*}
T(r,s)=\sum_{k=1}^{r+s+1}(-1)^{k-1}
\Big[\binom{2k}{2r+1}+\binom{2k}{2s+1}\Big(1-\frac{1}{2^{2k}}\Big)\Big]
\frac{1}{2^{2k}}\zeta(2k+1)
t(\{2\}^{r+s+1-k}).
\end{align*}
\end{abstract}

\author[J.M. Yu]{Jinmin Yu}
\address{Mathematical College, Sichuan University,
Chengdu 610064, P.R. China}
\email{jmyumath@163.com}
\author[S.F. Hong]{Shaofang Hong$^*$}
\address{Mathematical College, Sichuan University,
Chengdu 610064, P.R. China}
\email{sfhong@scu.edu.cn}
\thanks{$^*$S.F. Hong is the corresponding author and
was supported partially by National Science Foundation
of China \#12571007.}

\maketitle

\section{Introduction}
Let $l$, $k_1,k_2,\ldots, k_l$ be positive integers. For any multiple
index $\mathbf{s}=(s_1,s_2,\cdots,s_l)\in \mathbb{Z}_{\geq 1}^l$ with
$s_l\textgreater 1$, the {\it multiple zeta value (MZV for short)}
is defined by
\begin{align*}
\zeta(s_1,s_2,\cdots,s_l):=\sum_{1\leq k_1<\cdots<k_l}
\frac{1}{k_1^{s_1}k_2^{s_2}\cdots k_l^{s_l}}.
\end{align*}
The theory of MZVs was initially developed at the beginning of 1990s
independently by Hoffman \cite{[H-PJM]} and Zagier \cite{[Z-FECM]}.
The number $k:=s_1+s_2+\cdots+s_l$ is called the {\it weight} and $l$
the {\it depth} (or {\it length}) of MZVs. Obviously, there are
$\binom{k-2}{l-1}$ multiple zeta values of given weight $k$ and
depth $l$ when $0<l<k$. The simplest precise evaluations of MZVs
is given as follows
\begin{align}\label{eq1.1}
\zeta(\mathop{\underbrace{2,2,\ldots,2}}_{d})=\frac{\pi^{2d}}{(2d+1)!}.
\end{align}

A central problem is to investigate the algebraic structure of the
$\mathbb{Q}$-vector space $\mathfrak{Z}_k$ spanned by MZVs of weight
$k$. In this direction, Hoffman \cite{[H-JA]} proposed a conjecture
that every MZV of weight $k$ can be expressed as a $\mathbb{Q}$-linear
combination of those MZVs whose entries are only $2$'s and $3$'s.
This conjecture was later proved by Brown \cite{[B-ANN]} using motivic
method. In particular, Brown showed that the special MZVs
$$H(r,s):=\zeta(\mathop{\underbrace{2,2,\ldots,2}}_{r},3,
\mathop{\underbrace{2,2,\ldots,2}}_{s})$$
can be expressed as a
$\mathbb{Q}$-linear combination of products $\pi^{2a}\zeta(2b+1)$ with
$a+b=r+s+1$. However, Brown's motivic method does not provide the
explicit values of those coefficients in the linear combination of
products mentioned above. This challenge was accomplished by Zagier
\cite{[Z-AM]}, who showed that the following explicit formula:

\begin{theorem}\label{theorem1.1}
For any nonnegative integers $r$ and $s$, we have
\begin{align}\label{eq1.2'}
H(r,s)=2\sum_{k=1}^{r+s+1}(-1)^kc^{(k)}_{r,s}\zeta(2k+1)
\zeta(\mathop{\underbrace{2,2,\ldots,2}}_{r+s+1-k}),
\end{align}
where
$$
c^{(k)}_{r,s}:=\binom{2k}{2r+2}-\Big(1-\frac{1}{2^{2k}}\Big)\binom{2k}{2s+1}.
$$
\end{theorem}

The proof of Theorem \ref{theorem1.1} is given by computing the associated
generating functions of both sides in a closed form, and then showing they
are entire functions of exponential growth that agree at sufficiently many
points to force their equality. Later on, his proof was simplified by Li in
\cite{[L-MRL]}. In 2017, Hessami Pilehrood and Hessami Pilehrood \cite{[HHT-JMAA]}
gave an alternative proof of Theorem \ref{theorem1.1}.

In a similar way, Hoffman \cite{[H-CNTP]} defined the multiple $t$-values
(``odd variant'' of MZVs). For any multiple index 
$\mathbf{s}=(s_1,s_2,\cdots,s_l)\in \mathbb{Z}_{\geq 1}^l$,
the {\it multiple $t$-value } is defined by
\begin{align*}
t(s_1,s_2,...,s_l):=\sum_{1\leq k_1<k_2<...<k_l}
\frac{1}{(2k_1-1)^{s_1}(2k_2-1)^{s_2}...(2k_l-1)^{s_l}}
\end{align*}
where if the index is empty, we define the values $t(\emptyset):=1$.
As we stated above, first let us recall the the analogue formula
for the simplest multiple zeta value (\cite{[MK-PSPM],[Z-FM]})
\begin{align}\label{eq1.2}
t(\mathop{\underbrace{2,2,\ldots,2}}_{d})=\frac{\pi^{2d}}{2^{2d}(2d)!}.
\end{align}
By the method of establishing the generating functions, T. Murakami
\cite{[M-MA]} obtained an equivalent version of Theorem \ref{theorem1.1}.
Define $\widetilde{t}(k_1,k_2,\ldots,k_l)=2^{|k|}t(k_1,k_2,\ldots,k_l)$ 
and set
$$K(r,s):=\widetilde{t}(\mathop{\underbrace{2,2,\ldots,2}}_{r},3,
\mathop{\underbrace{2,2,\ldots,2}}_{s}),$$
$K(d):=\widetilde{t}(\{2\}^d)$,
then for all integers $r,s\geq 0$, he found and proved the following formula:
\begin{align}\label{eq1.3}
K(r,s)=2\sum_{k=1}^{r+s+1}(-1)^{k-1}d^{(k)}_{r,s}K(r+s-k+1)\zeta(2k+1),
\end{align}
where
$$
d^{(k)}_{r,s}:=\binom{2k}{2r+1}+\Big(1-\frac{1}{2^{2k}}\Big)\binom{2k}{2s+1}.
$$
Moreover, let
$$T(r,s):=t(\mathop{\underbrace{2,2,\ldots,2}}_{r},3,
\mathop{\underbrace{2,2,\ldots,2}}_{s}),$$
then \eqref{eq1.3} is equivalent to
\begin{theorem}\label{theorem1.2}
For any nonnegative integers $r$ and $s$, we have
\begin{align}\label{eq1.5}
T(r,s)=\sum_{k=1}^{r+s+1}(-1)^{k-1}d^{(k)}_{r,s}\frac{1}{2^{2k}}\zeta(2k+1)
t(\mathop{\underbrace{2,2,\ldots,2}}_{r+s+1-k}).
\end{align}
\end{theorem}

In a different approach, Lupu \cite{[L-MZ]} provided elementary proofs
of Zagier's formula (\ref{eq1.2'}) for the multiple zeta values and
Murakami's formula (\ref{eq1.5}) for the multiple $t$-values for the
case $s=0$. In the process, Lupu \cite{[L-MZ]} showed that
\begin{align*}
H(r,0)=\frac{-4\pi^{2r+2}}{(2r+2)!}\sum_{n=0}^{\infty}
\frac{\zeta(2n)}{(2n+2r+2)(2n+2r+3)\cdots(2n+2r+3)2^{2n}}
\end{align*}
and
\begin{align*}
T(r,0)=\frac{-2}{(2r+1)!}\left(\frac{\pi}{2}\right)^{2r+2}
\sum_{n=0}^{\infty}\frac{\zeta(2n)}{(2n+2r+1)(2n+2r+2)\cdots(2n+2r+2)2^{2n}}.
\end{align*}
Meanwhile, for
the general cases $s>0$, Lupu \cite{[L-MZ]} proposed two conjectures
stating that for any nonnegative integers $r$ and $s$, we have
\begin{align*}
H(r,s)=\frac{-4\pi^{2r+2s+2}}{(2r+2)!}\sum_{n=0}^{\infty}
\frac{\zeta(2n)}{(2n+2r+2)(2n+2r+3)\cdots(2n+2r+2s+3)2^{2n}}
\end{align*}
and
\begin{align*}
T(r,s)=\frac{-2}{(2r+1)!}\left(\frac{\pi}{2}\right)^{2r+2s+2}
\sum_{n=0}^{\infty}\frac{\zeta(2n)}{(2n+2r+1)(2n+2r+2)\cdots(2n+2r+2s+2)2^{2n}}.
\end{align*}
These two conjectures of Lupu were very recently proved to be true
by Lei, Yu and Hong \cite{[LYH]}. On the other hand, Lai, Lupu,
and Orr \cite{[LLO-PAMS]} presented elementary and direct proofs of
Zagier's formula for $H(r, s)$ and Murakami's formula for $T(r, s)$.

In this paper, our main goal is to give another proofs of Zagier's 
formula (\ref{eq1.2'}) for $H(r,s)$ and its odd variant
(i.e. Murakami's formula (\ref{eq1.5}) for $T(r,s)$). In fact, with the
help of the Lei-Yu-Hong expressions for $H(r,s)$ and $T(r,s)$ as well
as Lupu's result about rational zeta series involving Riemann zeta
values $\zeta(2n)$, and by establishing some identities about binomial
coefficients and a result about Kronecker symbol and arithmetic functions,
we present new elementary proofs of Zagier's formula (\ref{eq1.2'}) for
$H(r,s)$ and Murakami's formula (\ref{eq1.5}) for $T(r,s)$.

This paper is organized as follows. In Section 2 we supply several
preliminary lemmas that are needed in the proofs of Theorems
\ref{theorem1.1} and \ref{theorem1.2}. Sections 3 and 4 are
devoted to the proofs of Theorems \ref{theorem1.1} and
\ref{theorem1.2}, respectively.

\section{Preliminary lemmas}
In this section, we give some preliminary lemmas that are needed
in the proofs of our main results. We begin with the Lei-Yu-Hong
solutions \cite{[LYH]} to Lupu's conjectures \cite{[L-MZ]}.

\begin{lemma} \label{theorem1.3} (Theorem 1.4 of \cite{[LYH]})
For nonnegative integers $r$ and $s$, we have
\begin{align*}
H(r,s)=\frac{-4\pi^{2r+2s+2}}{(2r+2)!}\sum_{n=0}^{\infty}
\frac{\zeta(2n)}{(2n+2r+2)(2n+2r+3)\cdots(2n+2r+2s+3)2^{2n}}
\end{align*}
\end{lemma}

\begin{lemma}\label{theorem1.4}(Theorem 1.5 of \cite{[LYH]})
For nonnegative integers $r$ and $s$, we have
\begin{align*}
T(r,s)=\frac{-2}{(2r+1)!}\left(\frac{\pi}{2}\right)^{2r+2s+2}
\sum_{n=0}^{\infty}\frac{\zeta(2n)}{(2n+2r+1)(2n+2r+2)\cdots(2n+2r+2s+2)2^{2n}}.
\end{align*}
\end{lemma}

\begin{lemma}\label{lemma2.1}
For any positive integers $m$ and $j$ with $m>j$, we have
\begin{align}\label{eq2.1}
\prod_{i=j}^{m}\frac{1}{2n+2r+i}
=\frac{1}{(m-j)!}\sum_{i=j}^{m}(-1)^{i-j}\binom{m-j}{i-j}\frac{1}{2n+2r+i}.
\end{align}
\end{lemma}

\begin{proof}
We use the induction on $m$ to show Lemma 2.3.

First, let $m=j+1$. Then the left-hand side of \eqref{eq2.1} equals
\begin{align*}
\frac{1}{(2n+2r+j)(2n+2r+j+1)}=\frac{1}{2n+2r+j}-\frac{1}{2n+2r+j+1},
\end{align*}
and the right-hand side of \eqref{eq2.1} is equal to
\begin{align*}
\sum_{i=j}^{j+1}(-1)^{i-j}\binom{1}{i-j}\frac{1}{2n+2r+i}
=\frac{1}{2n+2r+j}-\frac{1}{2n+2r+j+1}.
\end{align*}
Hence \eqref{eq2.1} is true when $m=j+1$.

Now let $m>j+1$. Assume that  \eqref{eq2.1} is true for the $m-1$ case.
In what follows, we show  \eqref{eq2.1} holds for the $m$ case.
By the inductive hypothesis, we have
\begin{align*}
&\prod_{i=j}^{m}\frac{1}{2n+2r+i}\nonumber\\
=&\frac{1}{2n+2r+m}\prod_{i=j}^{m-1}\frac{1}{2n+2r+i}\nonumber\\
=&\frac{1}{2n+2r+m}\cdot\frac{1}{(m-j-1)!}
\sum_{i=j}^{m-1}(-1)^{i-j}\binom{m-j-1}{i-j}\frac{1}{2n+2r+i}\nonumber\\
=&\frac{1}{(m-j-1)!}\sum_{i=j}^{m-1}(-1)^{i-j}\binom{m-j-1}{i-j}
\frac{1}{m-i}\Big(\frac{1}{2n+2r+i}-\frac{1}{2n+2r+m}\Big)\nonumber\\
=&\frac{1}{(m-j-1)!}\sum_{i=j}^{m-1}(-1)^{i-j}\frac{(m-j-1)!}{(i-j)!(m-i-1)!}
\frac{1}{m-i}\Big(\frac{1}{2n+2r+i}-\frac{1}{2n+2r+m}\Big)\nonumber\\
=&\frac{1}{(m-j)!}\sum_{i=j}^{m-1}(-1)^{i-j}\frac{(m-j)!}{(i-j)!(m-i)!}
\Big(\frac{1}{2n+2r+i}-\frac{1}{2n+2r+m}\Big)\nonumber\\
=&\frac{1}{(m-j)!}\sum_{i=j}^{m-1}(-1)^{i-j}\binom{m-j}{i-j}
\Big(\frac{1}{2n+2r+i}-\frac{1}{2n+2r+m}\Big)\nonumber\\
=&\frac{1}{(m-j)!}\Big(\sum_{i=j}^{m-1}(-1)^{i-j}\binom{m-j}{i-j}
\frac{1}{2n+2r+i}-\frac{1}{2n+2r+m}\sum_{i=j}^{m-1}(-1)^{i-j}\binom{m-j}{i-j}\Big).
\end{align*}

Let us calculate the sum $\sum_{i=j}^{m-1}(-1)^{i-j}\binom{m-j}{i-j}$. One notices that
\begin{align}\label{eq2.2}
(-1)^{i-j}\binom{m-j}{i-j}=G(m,i+1)-G(m,i),
\end{align}
where
\begin{align*}
G(m,i):=(-1)^{i-j-1}\binom{m-j-1}{i-j-1}.
\end{align*}
Then summing both sides of equation \eqref{eq2.2} over $i$ from $j$ to $m-1$,
we obtain that
\begin{align*}
&\sum_{i=j}^{m-1}(-1)^{i-j}\binom{m-j}{i-j}\\
=&\sum_{i=j}^{m-1}(G(m,i+1)-G(m,i))\\
=& G(m,m)-G(m,j)\\
=& (-1)^{m-j-1}.
\end{align*}
It then follows that
\begin{align*}
\prod_{i=j}^{m}\frac{1}{2n+2r+i}
=&\frac{1}{(m-j)!}\Big(\sum_{i=j}^{m-1}(-1)^{i-j}\binom{m-j}{i-j}
\frac{1}{(2n+2r+i)}-\frac{(-1)^{m-j-1}}{2n+2r+m}\Big)\\
=&\frac{1}{(m-j)!}\sum_{i=j}^{m}(-1)^{i-j}\binom{m-j}{i-j}\frac{1}{(2n+2r+i)}
\end{align*}
as claimed for the $m$ case. So Lemma 2.3 is proved.
\end{proof}

\begin{lemma}\label{lemma2.2}
For any positive integers $s$, $r$ and $k$ with $k\le \frac{1}{2}(r+1)+s$, we have
\begin{align}\label{eq2.3}
\sum_{i=0}^{2s+1} (-1)^i \binom{2s+1}{i} \binom{r+i}{2k}
=-\binom{r}{2k-2s-1}
\end{align}
and
\begin{align}\label{eq2.4}
\sum_{i=2k-r}^{2s+1}(-1)^i \binom{2s+1}{i} \binom{r+i}{2k}
=-\binom{r}{2k-2s-1}
\end{align}
\end{lemma}
\begin{proof}
First of all, we introduce an auxiliary polynomial $F(x)$ as follows
$$
F(x):=\sum_{i=0}^{2s+1}(-1)^i \binom{2s+1}{i}(1+x)^{r+i}.
$$
Then
\begin{align}\label{eq2.5}
F(x)=&\sum_{i=0}^{2s+1}(-1)^i \binom{2s+1}{i}\sum_{j=0}^{r+i}\binom{r+i}{j}x^j\\
=&\sum_{j=0}^{r+2s+1}\Big(\sum_{i=j-r}^{2s+1}(-1)^i \binom{2s+1}{i}\binom{r+i}{j}\Big)x^j.\label{eq2.6}
\end{align}
and
\begin{align}\label{eq2.7}
F(x)=&(1+x)^{r}\sum_{i=0}^{2s+1}\binom{2s+1}{i}(-1)^i (1+x)^{i}\nonumber\\
=&(1+x)^{r}(1-(1+x))^{2s+1}\nonumber\\
=&-(1+x)^{r}x^{2s+1}\nonumber\\
=&-x^{2s+1}\sum_{j=0}^{r}\binom{r}{j}x^j\nonumber\\
=&-\sum_{j=0}^{r}\binom{r}{j}x^{2s+j+1}.
\end{align}
Letting $j=2k$ in \eqref{eq2.5} and picking $j=2k-2s-1$ in \eqref{eq2.7},
comparing the coefficient of $x^{2k}$ give us the desired identity \eqref{eq2.3}.
Likewise, letting $j=2k$ in \eqref{eq2.6} and picking $j=2k-2s-1$ in \eqref{eq2.7},
comparing the coefficient of $x^{2k}$ give us the desired identity \eqref{eq2.4}.

This concludes the proof of Lemma 2.4.
\end{proof}

The following result about rational zeta series involving the values $\zeta(2n)$
is due to Lupu and also needed in our proofs.

\begin{lemma}\label{lemma2.3} (Lemma 2.6 of \cite{[L-MZ]})
Let $p$ be a positive integer. Then
\begin{align*}
& \sum_{n=0}^{\infty}\frac{\zeta(2n)}{(2n+p)2^{2n}}\\
=&-\frac{1}{2}\log2-\sum_{k=1}^{\lfloor p/2\rfloor}\frac{p!(-1)^k(4^k-1)\zeta(2k+1)}
{2\cdot(p-2k)!(2\pi)^{2k}}-\delta_{\lfloor p/2\rfloor,p/2}\frac{p!(-1)^{p/2}
\zeta(p+1)}{2\pi^p},
\end{align*}
where $\delta_{i,j}$ is the Kronecker symbol defined by
$$
\delta_{i,j}:=
\left\{\begin{array}{ll}
\displaystyle 0,
&\textrm{ if \ $i\neq j$, }\\
\displaystyle 1,
&\textrm{ if \ $i=j$.}
\end{array}\right.
$$
\end{lemma}

\begin{lemma}\label{lemma2.4}
Let $I$ be a finite set of positive integers and let $f$
be an arithmetic function. Then for any $a\in I$, we have
$$
\sum_{k\in I} f(k)\delta_{a,k}=f(a).
$$
In particular, for any positive integer $a$, we have
$$
\sum_{k=1}^a f(k)\delta_{a,k}=f(a),
$$
and for any positive integer $b\ge a$, we have
$$
\sum_{k=a}^b f(k)\delta_{a,k}=f(a).
$$
\end{lemma}
\begin{proof}
By the definition of  the Kronecker symbol, one knows that $\delta_{a,a}=1$ and
$\delta_{a,k}=0$ if $k\in I\setminus\{a\}$. So
$$
\sum_{k\in I} f(k)\delta_{a,k}=\sum_{k\in I\setminus\{a\}} f(k)\delta_{a,k}+\delta_{a,a}f(a)=f(a)
$$
as expected.

This concludes the proof of Lemma 2.6.
\end{proof}

\section{Proof of Theorems \ref{theorem1.1}}
In this section, we use the lemmas presented in the  previous
section to show Theorem \ref{theorem1.1}.\\

\noindent {\it Proof of Theorem \ref{theorem1.1}.}
First of all, by Lemma \ref{theorem1.3} we know that for nonnegative
integers $r$ and $s$,
\begin{align*}
H(r,s)=\frac{-4\pi^{2r+2s+2}}{(2r+2)!}\sum_{n=0}^{\infty}
\frac{\zeta(2n)}{(2n+2r+2)(2n+2r+3)\cdots(2n+2r+2s+3)2^{2n}}.
\end{align*}
Let
\begin{align}\label{eq3.2}
A_{r,s}:=\sum_{n=0}^{\infty}\frac{\zeta(2n)}
{(2n+2r+2)(2n+2r+3)\cdots(2n+2r+2s+3)2^{2n}}.
\end{align}
Then
\begin{align}\label{eq3.1}
H(r,s)=\frac{-4\pi^{2r+2s+2}}{(2r+2)!}A_{r,s}.
\end{align}
Applying Lemma \ref{lemma2.1} yields that
\begin{align}\label{eq3.3}
A_{r,s}=&\sum_{n=0}^\infty \frac{\zeta(2n)}{2^{2n}}\prod_{i=2}^{2s+3}\frac{1}{2n+2r+i}\nonumber\\
=&\frac{1}{(2s+1)!}\sum_{i=2}^{2s+3}(-1)^i\binom{2s+1}{i-2}
\sum_{n=0}^{\infty}\frac{\zeta(2n)}{(2n+2r+i)2^{2n}}\nonumber\\
:=&\frac{1}{(2s+1)!} B_{r,s}.
\end{align}

Secondly, let us calculate the sum $B_{r,s}$. By the definition of Kronecker's
symbol and using Lemma \ref{lemma2.3} with $p$ being replaced by $2r+i$, we derive that
\begin{align*}
B_{r,s}=&-\frac{1}{2}\sum_{i=2}^{2s+3}(-1)^i\binom{2s+1}{i-2}\Big(\log 2+
\sum_{k=1}^{\lfloor r+i/2\rfloor}\frac{(2r+i)!(-1)^k(4^k-1)\zeta(2k+1)}
{(2r+i-2k)!(2\pi)^{2k}}\nonumber\\
&+\delta_{\lfloor r+i/2\rfloor,r+i/2}\frac{(2r+i)!(-1)^{r+i/2}
\zeta(2r+i+1)}{\pi^{2r+i}}\Big)\nonumber\\
=&-\frac{\log 2}{2}\sum_{i=0}^{2s+1}(-1)^i\binom{2s+1}{i}
-\frac{1}{2}\sum_{i=2}^{2s+3}(-1)^i\binom{2s+1}{i-2}(2r+i)!\nonumber\\
&\times\Big(\sum_{k=1}^{\lfloor r+i/2\rfloor}\frac{(-1)^k(4^k-1)\zeta(2k+1)}
{(2r+i-2k)!(2\pi)^{2k}}+\delta_{\lfloor r+i/2\rfloor,r+i/2}\frac{(-1)^{r+i/2}
\zeta(2r+i+1)}{\pi^{2r+i}}\Big)\nonumber\\
=&-\frac{1}{2}\mathop{\sum_{i=2}^{2s+3}}_{2|i}(-1)^i\binom{2s+1}{i-2}(2r+i)!
\Big(\sum_{k=1}^{r+i/2}\frac{(-1)^k(4^k-1)\zeta(2k+1)}
{(2r+i-2k)!(2\pi)^{2k}}\nonumber\\
&+\delta_{r+i/2,r+i/2}\frac{(-1)^{r+i/2}\zeta(2r+i+1)}{\pi^{2r+i}}\Big)\nonumber\\
&-\frac{1}{2}\mathop{\sum_{i=2}^{2s+3}}_{2\nmid i}(-1)^i\binom{2s+1}{i-2}(2r+i)!
\sum_{k=1}^{r+(i-1)/2}\frac{(-1)^k(4^k-1)\zeta(2k+1)}{(2r+i-2k)!(2\pi)^{2k}}\nonumber\\
=&-\frac{1}{2}\sum_{i'=1}^{s+1}\binom{2s+1}{2i'-2}(2r+2i')!
\sum_{k=1}^{r+i'}\frac{(-1)^k(4^k-1)\zeta(2k+1)}{(2r+2i'-2k)!(2\pi)^{2k}}\nonumber\\
&+\frac{1}{2}\sum_{i'=1}^{s+1}\binom{2s+1}{2i'-1}(2r+2i'+1)!
\sum_{k=1}^{r+i'}\frac{(-1)^k(4^k-1)\zeta(2k+1)}{(2r+2i'+1-2k)!(2\pi)^{2k}}\nonumber\\
&-\frac{1}{2}\sum_{i'=1}^{s+1}\binom{2s+1}{2i'-2}
\delta_{r+i',r+i'}\frac{(2r+2i')!(-1)^{r+i'}
\zeta(2r+2i'+1)}{\pi^{2r+2i'}}.
\end{align*}
Letting in Lemma \ref{lemma2.4}
$$f(k):=\frac{(2k)!(-1)^k\zeta(2k+1)}{\pi^{2k}}$$
gives us that
\begin{align*}
\delta_{r+i',r+i'}\frac{(2r+2i')!(-1)^{r+i'}\zeta(2r+2i'+1)}{\pi^{2r+2i'}}
=\sum_{k=1}^{r+i'}\delta_{r+i',k}\frac{(2k)!(-1)^k\zeta(2k+1)}{\pi^{2k}}.
\end{align*}
It then follows that
\begin{align*}
& B_{r,s}\nonumber\\
=&-\frac{1}{2}\sum_{i'=1}^{s+1}\binom{2s+1}{2i'-2}(2r+2i')!
\sum_{k=1}^{r+i'}\frac{(-1)^k(4^k-1)\zeta(2k+1)}{(2r+2i'-2k)!(2\pi)^{2k}}\nonumber\\
&+\frac{1}{2}\sum_{i'=1}^{s+1}\binom{2s+1}{2i'-1}(2r+2i'+1)!
\sum_{k=1}^{r+i'}\frac{(-1)^k(4^k-1)\zeta(2k+1)}{(2r+2i'+1-2k)!(2\pi)^{2k}}\nonumber\\
&-\frac{1}{2}\sum_{i'=1}^{s+1}\binom{2s+1}{2i'-2}\sum_{k=1}^{r+i'}\delta_{r+i',k}
\frac{(2k)!(-1)^k\zeta(2k+1)}{\pi^{2k}}\nonumber\\
=&-\frac{1}{2}\sum_{k=1}^{r+s+1}\frac{(-1)^k(4^k-1)\zeta(2k+1)}{(2\pi)^{2k}}
\sum_{i'=\max(1,k-r)}^{s+1}\binom{2s+1}{2i'-2}\frac{(2r+2i')!}{(2r+2i'-2k)!}\nonumber\\
&+\frac{1}{2}\sum_{k=1}^{r+s+1}\frac{(-1)^k(4^k-1)\zeta(2k+1)}{(2\pi)^{2k}}
\sum_{i'=\max(1,k-r)}^{s+1}\binom{2s+1}{2i'-1}\frac{(2r+2i'+1)!}{(2r+2i'+1-2k)!}\nonumber\\
&-\frac{1}{2}\sum_{k=1}^{r+s+1}\frac{(2k)!(-1)^k\zeta(2k+1)}{\pi^{2k}}
\sum_{i'=\max(1,k-r)}^{s+1}\binom{2s+1}{2i'-2}\delta_{r+i',k}\nonumber\\
=&-\frac{1}{2}\sum_{k=1}^{r+s+1}\frac{(-1)^k(4^k-1)\zeta(2k+1)}{(2\pi)^{2k}}
\sum_{i'=\max(1,k-r)}^{s+1}\binom{2s+1}{2i'-2}\frac{(2r+2i')!}{(2r+2i'-2k)!}\nonumber\\
&+\frac{1}{2}\sum_{k=1}^{r+s+1}\frac{(-1)^k(4^k-1)\zeta(2k+1)}{(2\pi)^{2k}}
\sum_{i'=\max(1,k-r)}^{s+1}\binom{2s+1}{2i'-1}\frac{(2r+2i'+1)!}{(2r+2i'+1-2k)!}\nonumber\\
&-\frac{1}{2}\sum_{k=r+1}^{r+s+1}\frac{(2k)!(-1)^k\zeta(2k+1)}{\pi^{2k}}
\sum_{i'=k-r}^{s+1}\binom{2s+1}{2i'-2}\delta_{r+i',k}
\text{\ (since $\delta_{r+i',k}=0$ for $1\le k \le r$)}\nonumber\\
=&-\frac{1}{2}\sum_{k=1}^{r+s+1}\frac{(-1)^k(4^k-1)\zeta(2k+1)}{(2\pi)^{2k}}\nonumber\\
&\times\sum_{i'=\max(1,k-r)}^{s+1}\Big(\binom{2s+1}{2i'-2}\frac{(2r+2i')!}{(2r+2i'-2k)!}
-\binom{2s+1}{2i'-1}\frac{(2r+2i'+1)!}{(2r+2i'+1-2k)!}\Big)\nonumber\\
&-\frac{1}{2}\sum_{k=r+1}^{r+s+1}\frac{(2k)!(-1)^k\zeta(2k+1)}{\pi^{2k}}
\sum_{i'=k-r}^{s+1}\binom{2s+1}{2i'-2}\delta_{r+i',k}\nonumber\\
=&-\frac{1}{2}\sum_{k=1}^{r+s+1}\frac{(-1)^k(4^k-1)\zeta(2k+1)(2k)!}{(2\pi)^{2k}}\nonumber\\
&\times\sum_{i'=\max(1,k-r)}^{s+1}\Big(\binom{2s+1}{2i'-2}\binom{2r+2i'}{2k}
-\binom{2s+1}{2i'-1}\binom{2r+2i'+1}{2k}\Big)\nonumber\\
&-\frac{1}{2}\sum_{k=r+1}^{r+s+1}\frac{(2k)!(-1)^k\zeta(2k+1)}{\pi^{2k}}
\sum_{i'=k-r}^{s+1}\binom{2s+1}{2i'-2}\delta_{r+i',k}.
\end{align*}

Notice that
\begin{align*}
&\sum_{i'=\max(1,k-r)}^{s+1}
\Big(\binom{2s+1}{2i'-2}\binom{2r+2i'}{2k}
-\binom{2s+1}{2i'-1}\binom{2r+2i'+1}{2k}\Big)\nonumber\\
=&\mathop{\sum_{i'=\max(2,2k-2r)-2}^{2s}}_{2|i}
\binom{2s+1}{i'}\binom{2r+i'+2}{2k}
-\mathop{\sum_{i'=\max(2,2k-2r)-1}^{2s+1}}_{2\nmid i}
\binom{2s+1}{i'}\binom{2r+i'+2}{2k}\nonumber\\
=&\sum_{i'=\max(2,2k-2r)-2}^{2s+1}(-1)^{i'}
\binom{2s+1}{i'}\binom{2r+i'+2}{2k}.
\end{align*}
Moreover, if $1\leq k\leq r$, then $k-r\leq 0$, and so $\max(2,2k-2r)=2$.
This concludes that
\begin{align}\label{eq3.4}
& B_{r,s}\nonumber\\
=&-\frac{1}{2}\sum_{k=1}^{r+s+1}\frac{(-1)^k(4^k-1)\zeta(2k+1)(2k)!}
{(2\pi)^{2k}}\sum_{i'=\max(2,2k-2r)-2}^{2s+1}(-1)^{i'}
\binom{2s+1}{i'}\binom{2r+i'+2}{2k}\nonumber\\
&-\frac{1}{2}\sum_{k=r+1}^{r+s+1}\frac{(2k)!(-1)^k\zeta(2k+1)}{\pi^{2k}}
\sum_{i'=k-r}^{s+1}\binom{2s+1}{2i'-2}\delta_{r+i',k}\nonumber\\
=&-\frac{1}{2}\sum_{k=1}^{r}\frac{(-1)^k(4^k-1)\zeta(2k+1)(2k)!}
{(2\pi)^{2k}}\sum_{i'=0}^{2s+1}(-1)^{i'}
\binom{2s+1}{i'}\binom{2r+i'+2}{2k}\nonumber\\
&-\frac{1}{2}\sum_{k=r+1}^{r+s+1}\frac{(-1)^k(4^k-1)\zeta(2k+1)(2k)!}
{(2\pi)^{2k}}\sum_{i'=2k-2r-2}^{2s+1}(-1)^{i'}
\binom{2s+1}{i'}\binom{2r+i'+2}{2k}\nonumber\\
&-\frac{1}{2}\sum_{k=r+1}^{r+s+1}\frac{(2k)!(-1)^k\zeta(2k+1)}{\pi^{2k}}
\binom{2s+1}{2k-2r-2}\delta_{k,k}.
\end{align}

However, by Lemma \ref{lemma2.2}, one has
\begin{align}\label{eq3.5}
\sum_{i'=0}^{2s+1}(-1)^{i'}\binom{2s+1}{i'}\binom{2r+i'+2}{2r+i'+2-2k}
=-\binom{2r+2}{2k-2s-1}
\end{align}
and
\begin{align}\label{eq3.6}
\sum_{i'=2k-2r-2}^{2s+1}(-1)^{i'}\binom{2s+1}{i'}\binom{2r+i'+2}{2r+i'+2-2k}=-\binom{2r+2}{2k-2s-1}.
\end{align}
Then putting \eqref{eq3.5} and \eqref{eq3.6} into \eqref{eq3.4}, one arrives at
\begin{align}\label{eq3.7}
& B_{r,s}\nonumber\\
=&\frac{1}{2}\sum_{k=1}^{r}\frac{(-1)^k(4^k-1)\zeta(2k+1)(2k)!}
{(2\pi)^{2k}}\binom{2r+2}{2k-2s-1}\nonumber\\
&+\frac{1}{2}\sum_{k=r+1}^{r+s+1}\frac{(-1)^k(4^k-1)\zeta(2k+1)(2k)!}
{(2\pi)^{2k}}\binom{2r+2}{2k-2s-1}\nonumber\\
&-\frac{1}{2}\sum_{k=r+1}^{r+s+1}\frac{(2k)!(-1)^k\zeta(2k+1)}{\pi^{2k}}
\binom{2s+1}{2k-2r-2}\nonumber\\
=&\frac{1}{2}\sum_{k=1}^{r+s+1}\frac{(-1)^k(4^k-1)\zeta(2k+1)(2k)!}
{(2\pi)^{2k}}\binom{2r+2}{2k-2s-1}\nonumber\\
&-\frac{1}{2}\sum_{k=r+1}^{r+s+1}
\frac{(2k)!(-1)^k\zeta(2k+1)}{\pi^{2k}}\binom{2s+1}{2k-2r-2}\nonumber\\
=&\frac{1}{2}\sum_{k=1}^{r+s+1}\frac{(-1)^k(4^k-1)\zeta(2k+1)(2k)!}
{(2\pi)^{2k}}\binom{2r+2}{2k-2s-1}\nonumber\\
&-\frac{1}{2}\sum_{k=1}^{r+s+1}
\frac{(2k)!(-1)^k\zeta(2k+1)}{\pi^{2k}}\binom{2s+1}{2k-2r-2}
\text{\ (since $\binom{2s+1}{2k-2r-2}:=0$ if $1\le k \le r$).}
\end{align}

Now substituting \eqref{eq3.7} into \eqref{eq3.3}, we obtain that
\begin{align}\label{eq3.8}
A_{r,s}=&\frac{1}{(2s+1)!} B_{r,s}\nonumber\\
=&\frac{1}{2\cdot(2s+1)!}\sum_{k=1}^{r+s+1}\frac{(-1)^k(4^k-1)\zeta(2k+1)(2k)!}
{(2\pi)^{2k}}\binom{2r+2}{2k-2s-1}\nonumber\\
&-\frac{1}{2\cdot(2s+1)!}\sum_{k=1}^{r+s+1}\frac{(2k)!(-1)^k\zeta(2k+1)}{\pi^{2k}}
\binom{2s+1}{2k-2r-2}.
\end{align}
Then putting \eqref{eq3.8} into \eqref{eq3.1}, one gets that
\begin{align}\label{eq3.9}
& H(r,s)\nonumber\\
=&\frac{-4\pi^{2r+2s+2}}{(2r+2)!}A_{r,s}\nonumber\\
=&-\frac{4\pi^{2r+2s+2}}{(2r+2)!}\frac{1}{2\cdot(2s+1)!}
\sum_{k=1}^{r+s+1}\frac{(-1)^k(4^k-1)\zeta(2k+1)(2k)!}
{(2\pi)^{2k}}\binom{2r+2}{2k-2s-1}\nonumber\\
&+\frac{4\pi^{2r+2s+2}}{(2r+2)!}\frac{1}{2\cdot(2s+1)!}
\sum_{k=1}^{r+s+1}\frac{(2k)!(-1)^k\zeta(2k+1)}{\pi^{2k}}
\binom{2s+1}{2k-2r-2}\nonumber\\
=&-\frac{2\pi^{2r+2s+2}}{(2r+2)!}\frac{1}{(2s+1)!}
\sum_{k=1}^{r+s+1}\frac{(-1)^{k}(4^k-1)\zeta(2k+1)(2k)!}{(2\pi)^{2k}}
\frac{(2r+2)!}{(2k-2s-1)!(2r+2s+3-2k)!}\nonumber\\
&+\frac{2\pi^{2r+2s+2}}{(2r+2)!}\cdot\frac{1}{(2s+1)!}
\sum_{k=1}^{r+s+1}\frac{(2k)!(-1)^k}{\pi^{2k}}
\frac{(2s+1)!}{(2k-2r-2)!(2r+2s+3-2k)!}\zeta(2k+1)\nonumber\\
=&-2\sum_{k=1}^{r+s+1}\frac{(2k)!}{(2s+1)!(2k-2s-1)!}\cdot
\frac{\pi^{2r+2s+2-2k}}{(2r+2s+3-2k)!}
(-1)^k\Big(1-\frac{1}{4^k}\Big)\zeta(2k+1)\nonumber\\
&+2\sum_{k=1}^{r+s+1}\frac{\pi^{2r+2s+2-2k}}{(2r+2s+3-2k)!}
\frac{(2k)!}{(2r+2)!(2k-2r-2)!}(-1)^k\zeta(2k+1)\nonumber\\
=&-2\sum_{k=1}^{r+s+1}\binom{2k}{2s+1}
\frac{\pi^{2r+2s+2-2k}}{(2r+2s+3-2k)!}(-1)^k\Big(1-\frac{1}{4^k}\Big)\zeta(2k+1)\nonumber\\
&+2\sum_{k=1}^{r+s+1}\frac{\pi^{2r+2s+2-2k}}{(2r+2s+3-2k)!}
\binom{2k}{2r+2}(-1)^k\zeta(2k+1)\nonumber\\
=&2 \sum_{k=1}^{r+s+1}(-1)^k\left[\binom{2k}{2r+2}-\Big(1-\frac{1}{2^{2k}}\Big)
\binom{2k}{2s+1}\right]\frac{\pi^{2r+2s+2-2k}}{(2r+2s+3-2k)!}\zeta(2k+1).
\end{align}

Finally, noticing that by \eqref{eq1.1}, we have
\begin{align}\label{eq3.10}
\frac{\pi^{2r+2s+2-2k}}{(2r+2s+3-2k)!}=\zeta(\mathop{\underbrace{2,2,\cdots,2}}_{r+s+1-k}).
\end{align}
so putting \eqref{eq3.10} into \eqref{eq3.9} gives that
\begin{align*}
H(r,s)
=&2\sum_{k=1}^{r+s+1}(-1)^k
\left[\binom{2k}{2r+2}-\Big(1-\frac{1}{2^{2k}}\Big)\binom{2k}{2s+1}\right]
\zeta(\mathop{\underbrace{2,2,\cdots,2}}_{r+s+1-k})\zeta(2k+1)\\
=& 2\sum_{k=1}^{r+s+1}(-1)^kc^{(k)}_{r,s}\zeta(2k+1)
\zeta(\mathop{\underbrace{2,2,\ldots,2}}_{r+s+1-k})
\end{align*}
as expected.

This finishes the proof of Theorem \ref{theorem1.1}.   \hfill$\Box$\\

\section{Proof of Theorems \ref{theorem1.2}}
In this section, we use the lemmas presented in Section 2 to give the
proof of Theorem \ref{theorem1.2}.\\

\noindent {\it Proof of Theorem \ref{theorem1.2}.}
First of all, by Theorem \ref{theorem1.4} we know that for nonnegative
integers $r$ and $s$,
\begin{align*}
T(r,s)=\frac{-2}{(2r+1)!}\left(\frac{\pi}{2}\right)^{2r+2s+2}
\sum_{n=0}^{\infty}\frac{\zeta(2n)}{(2n+2r+1)(2n+2r+2)\cdots(2n+2r+2s+2)2^{2n}}.
\end{align*}

Let
\begin{align}\label{eq4.1}
&C_{r,s}:=\sum_{n=0}^{\infty}\frac{\zeta(2n)}
{(2n+2r+1)(2n+2r+2)\cdots(2n+2r+2s+2)2^{2n}}
\end{align}
Then
\begin{align}\label{eq4.2}
T(r,s)=\frac{-2}{(2r+1)!}\left(\frac{\pi}{2}\right)^{2r+2s+2}C_{r,s}.
\end{align}
Now we deal with the sum $C_{r,s}$. Applying Lemma \ref{lemma2.1} yields that
\begin{align}\label{eq4.3}
C_{r,s}
=&\sum_{n=0}^{\infty}\frac{\zeta(2n)}{2^{2n}}
\prod_{i=1}^{2s+2}\frac{1}{2n+2r+i}\nonumber\\
=&\sum_{n=0}^{\infty}\frac{\zeta(2n)}{2^{2n}}
\frac{1}{(2s+1)!}\sum_{i=1}^{2s+2}(-1)^{i-1}
\binom{2s+1}{i-1}\frac{1}{2n+2r+i}\nonumber\\
=&\frac{1}{(2s+1)!}\sum_{i=1}^{2s+2}(-1)^{i-1}\binom{2s+1}{i-1}
\sum_{n=0}^{\infty}\frac{\zeta(2n)}{(2n+2r+i)2^{2n}}\nonumber\\
:=&\frac{1}{(2s+1)!}D_{r,s}.
\end{align}

Secondly, let us calculate the sum $D_{r,s}$. By the definition of Kronecker's
symbol and using Lemma \ref{lemma2.3} with $p$ replaced by $2r+i$, we derive that
\begin{align}\label{4.4'}
D_{r,s}
=&-\frac{1}{2}\sum_{i=1}^{2s+2}(-1)^{i-1}\binom{2s+1}{i-1}
\Big(\log 2+\sum_{k=1}^{\lfloor r+i/2\rfloor}\frac{(2r+i)!(-1)^k(4^k-1)\zeta(2k+1)}
{(2r+i-2k)!(2\pi)^{2k}}\nonumber\\
&+\delta_{\lfloor r+i/2\rfloor,r+i/2}
\frac{(2r+i)!(-1)^{r+i/2}\zeta(2r+i+1)}{\pi^{2r+i}}\Big)\nonumber\\
=&-\frac{\log2}{2}\sum_{i=1}^{2s+2}(-1)^{i-1}\binom{2s+1}{i-1}
-\frac{1}{2}\sum_{i=1}^{2s+2}(-1)^{i-1}\binom{2s+1}{i-1}(2r+i)!\nonumber\\
&\times\Big(\sum_{k=1}^{\lfloor r+i/2\rfloor}\frac{(-1)^k(4^k-1)\zeta(2k+1)}
{(2r+i-2k)!(2\pi)^{2k}}+\delta_{\lfloor r+i/2\rfloor,r+i/2}
\frac{(-1)^{r+i/2}\zeta(2r+i+1)}{\pi^{2r+i}}\Big)\nonumber\\
=&\frac{1}{2}\mathop{\sum_{i=1}^{2s+2}}_{2|i}\binom{2s+1}{i-1}
(2r+i)!\Big(\sum_{k=1}^{r+i/2}\frac{(-1)^k(4^k-1)\zeta(2k+1)}
{(2r+i-2k)!(2\pi)^{2k}}\nonumber\\
&+\delta_{r+i/2,r+i/2}\frac{(-1)^{r+i/2}\zeta(2r+i+1)}{\pi^{2r+i}}\Big)\nonumber\\
&-\frac{1}{2}\mathop{\sum_{i=1}^{2s+2}}_{2\nmid i}\binom{2s+1}{i-1}
(2r+i)!\sum_{k=1}^{r+(i-1)/2}\frac{(-1)^k(4^k-1)\zeta(2k+1)}
{(2r+i-2k)!(2\pi)^{2k}}\nonumber\\
=&\frac{1}{2}\sum_{i'=1}^{s+1}\binom{2s+1}{2i'-1}(2r+2i')!
\sum_{k=1}^{r+i'}\frac{(-1)^k(4^k-1)\zeta(2k+1)}
{(2r+2i'-2k)!(2\pi)^{2k}}\nonumber\\
&+\frac{1}{2}\sum_{i'=1}^{s+1}\binom{2s+1}{2i'-1}
\delta_{r+i',r+i'}\frac{(2r+2i')!(-1)^{r+i'}\zeta(2r+2i'+1)}{\pi^{2r+2i'}}\nonumber\\
&-\frac{1}{2}\sum_{i'=1}^{s+1}\binom{2s+1}{2i'-2}(2r+2i'-1)!
\sum_{k=1}^{r+i'-1}\frac{(-1)^k(4^k-1)\zeta(2k+1)}
{(2r+2i'-1-2k)!(2\pi)^{2k}}.
\end{align}
Picking in Lemma \ref{lemma2.4}
$$
f(k):=\frac{(2k)!(-1)k\zeta(2k+1)}{\pi^{2k}},
$$
one obtains that
\begin{align}\label{4.5'}
\delta_{r+i',r+i'}\frac{(2r+2i')!(-1)^{r+i'}\zeta(2r+2i'+1)}{\pi^{2r+2i'}}
=\sum_{k=1}^{r+i'}\delta_{r+i',k}\frac{(2k)!(-1)^k\zeta(2k+1)}{\pi^{2k}}.
\end{align}
Putting \eqref{4.5'} into \eqref{4.4'}, it follows that
\begin{align*}
D_{r,s}=&\frac{1}{2}\sum_{i'=1}^{s+1}\binom{2s+1}{2i'-1}(2r+2i')!
\sum_{k=1}^{r+i'}\frac{(-1)^k(4^k-1)\zeta(2k+1)}
{(2r+2i'-2k)!(2\pi)^{2k}}\nonumber\\
&+\frac{1}{2}\sum_{i'=1}^{s+1}\binom{2s+1}{2i'-1}
\sum_{k=1}^{r+i'}\delta_{r+i',k}\frac{(2k)!(-1)^k\zeta(2k+1)}{\pi^{2k}}\nonumber\\
&-\frac{1}{2}\sum_{i'=1}^{s+1}\binom{2s+1}{2i'-2}(2r+2i'-1)!
\sum_{k=1}^{r+i'-1}\frac{(-1)^k(4^k-1)\zeta(2k+1)}
{(2r+2i'-1-2k)!(2\pi)^{2k}}\nonumber\\
=&\frac{1}{2}\sum_{i'=1}^{s+1}\binom{2s+1}{2i'-1}
\sum_{k=1}^{r+i'}\frac{(-1)^k(4^k-1)\zeta(2k+1)(2k)!}
{(2\pi)^{2k}}\frac{(2r+2i')!}{(2k)!(2r+2i'-2k)!}\nonumber\\
&+\frac{1}{2}\sum_{i'=1}^{s+1}\binom{2s+1}{2i'-1}
\sum_{k=1}^{r+i'}\delta_{r+i',k}\frac{(2k)!(-1)^k\zeta(2k+1)}{\pi^{2k}}\nonumber\\
&-\frac{1}{2}\sum_{i'=1}^{s+1}\binom{2s+1}{2i'-2}
\sum_{k=1}^{r+i'-1}\frac{(-1)^k(4^k-1)\zeta(2k+1)(2k)!}
{(2\pi)^{2k}}\frac{(2r+2i'-1)!}{(2k)!(2r+2i'-1-2k)!}\nonumber\\
=&\frac{1}{2}\sum_{i'=1}^{s+1}\binom{2s+1}{2i'-1}
\sum_{k=1}^{r+i'}\frac{(-1)^k(4^k-1)\zeta(2k+1)(2k)!}
{(2\pi)^{2k}}\binom{2r+2i'}{2k}\nonumber\\
&+\frac{1}{2}\sum_{i'=1}^{s+1}\binom{2s+1}{2i'-1}
\sum_{k=1}^{r+i'}\delta_{r+i',k}\frac{(2k)!(-1)^k\zeta(2k+1)}{\pi^{2k}}\nonumber\\
&-\frac{1}{2}\sum_{i'=1}^{s+1}\binom{2s+1}{2i'-2}
\sum_{k=1}^{r+i'-1}\frac{(-1)^k(4^k-1)\zeta(2k+1)(2k)!}
{(2\pi)^{2k}}\binom{2r+2i'-1}{2k}\nonumber\\
=&\frac{1}{2}\sum_{i'=1}^{s+1}\binom{2s+1}{2i'-1}
\sum_{k=1}^{r+i'}\frac{(-1)^k(4^k-1)\zeta(2k+1)(2k)!}
{(2\pi)^{2k}}\binom{2r+2i'}{2k}\nonumber\\
&+\frac{1}{2}\sum_{i'=1}^{s+1}\binom{2s+1}{2i'-1}
\sum_{k=1}^{r+i'}\delta_{r+i',k}\frac{(2k)!(-1)^k\zeta(2k+1)}{\pi^{2k}}\nonumber\\
&-\frac{1}{2}\sum_{i'=1}^{s+1}\binom{2s+1}{2i'-2}
\sum_{k=1}^{r+i'}\frac{(-1)^k(4^k-1)\zeta(2k+1)(2k)!}
{(2\pi)^{2k}}\binom{2r+2i'-1}{2k}\nonumber\\
&\text{\ (since $\binom{2r+2i'-1}{2k}=0$ for $k=r+i'$)}\nonumber\\
=&\frac{1}{2}\sum_{k=1}^{r+s+1}\frac{(-1)^k(4^k-1)\zeta(2k+1)(2k)!}
{(2\pi)^{2k}}\sum_{i'=\max(1,k-r)}^{s+1}\binom{2s+1}{2i'-1}
\binom{2r+2i'}{2k}\nonumber\\
&+\frac{1}{2}\sum_{k=1}^{r+s+1}\frac{(2k)!(-1)^k\zeta(2k+1)}{\pi^{2k}}
\sum_{i'=\max(1,k-r)}^{s+1}\binom{2s+1}{2i'-1}\delta_{r+i',k}\nonumber\\
&-\frac{1}{2}\sum_{k=1}^{r+s+1}\frac{(-1)^k(4^k-1)\zeta(2k+1)(2k)!}
{(2\pi)^{2k}}\sum_{i'=\max(1,k-r)}^{s+1}\binom{2s+1}{2i'-2}
\binom{2r+2i'-1}{2k}\nonumber\\
=&\frac{1}{2}\sum_{k=1}^{r+s+1}\frac{(-1)^k(4^k-1)\zeta(2k+1)(2k)!}
{(2\pi)^{2k}}\sum_{i'=\max(1,k-r)}^{s+1}\binom{2s+1}{2i'-1}
\binom{2r+2i'}{2k}\nonumber\\
&+\frac{1}{2}\sum_{k=r+1}^{r+s+1}\frac{(2k)!(-1)^k\zeta(2k+1)}{\pi^{2k}}
\sum_{i'=k-r}^{s+1}\binom{2s+1}{2i'-1}\delta_{r+i',k}
\text{\ (since $\delta_{r+i',k}=0$ for $1\le k \le r$)}\nonumber\\
&-\frac{1}{2}\sum_{k=1}^{r+s+1}\frac{(-1)^k(4^k-1)\zeta(2k+1)(2k)!}
{(2\pi)^{2k}}\sum_{i=\max(1,k-r)}^{s+1}\binom{2s+1}{2i'-2}
\binom{2r+2i'-1}{2k}\nonumber\\
=&\frac{1}{2}\sum_{k=1}^{r+s+1}\frac{(-1)^k(4^k-1)\zeta(2k+1)(2k)!}
{(2\pi)^{2k}}\sum_{i'=\max(1,k-r)}^{s+1}\binom{2s+1}{2i'-1}
\binom{2r+2i'}{2k}\nonumber\\
&-\frac{1}{2}\sum_{k=1}^{r+s+1}\frac{(-1)^k(4^k-1)\zeta(2k+1)(2k)!}
{(2\pi)^{2k}}\sum_{i'=\max(1,k-r)}^{s+1}\binom{2s+1}{2i'-2}
\binom{2r+2i'-1}{2k}\nonumber\\
&+\frac{1}{2}\sum_{k=r+1}^{r+s+1}\frac{(2k)!(-1)^k\zeta(2k+1)}{\pi^{2k}}
\sum_{i'=k-r}^{s+1}\binom{2s+1}{2i'-1}\delta_{r+i',k}\nonumber\\
=&\frac{1}{2}\sum_{k=1}^{r+s+1}\frac{(-1)^k(4^k-1)\zeta(2k+1)(2k)!}{(2\pi)^{2k}}\nonumber\\
&\times\sum_{i'=\max(1,k-r)}^{s+1}\Big(\binom{2s+1}{2i'-1}\binom{2r+2i'}{2k}
-\binom{2s+1}{2i'-2}\binom{2r+2i'-1}{2k}\Big)\nonumber\\
&+\frac{1}{2}\sum_{k=1}^{r+s+1}\frac{(2k)!(-1)^k\zeta(2k+1)}{\pi^{2k}}
\binom{2s+1}{2k-2r-1}
\text{\ (since $\binom{2s+1}{2k-2r-1}=0$ if $1\le k\le r$)}.
\end{align*}

Noting that
\begin{align*}
&\sum_{i'=\max(1,k-r)}^{s+1}\Big(\binom{2s+1}{2i'-1}\binom{2r+2i'}{2k}
-\binom{2s+1}{2i'-2}\binom{2r+2i'-1}{2k}\Big)\nonumber\\
=&\mathop{\sum_{i'=\max(2,2k-2r)-1}^{2s+1}}_{2\nmid i'}
\binom{2s+1}{i'}\binom{2r+i'+1}{2k}
-\mathop{\sum_{i'=\max(2,2k-2r)-2}^{2s}}_{2|i'}
\binom{2s+1}{i'}\binom{2r+i'+1}{2k}\nonumber\\
=&\sum_{i'=\max(2,2k-2r)-2}^{2s+1}(-1)^{i+1}\binom{2s+1}{i'}\binom{2r+i'+1}{2k},
\end{align*}
and that if $1\leq k\leq r$, then $k-r\leq 0$, and so $\max(1,k-r)=1$,
it then follows that
\begin{align}\label{eq4.4}
& D_{r,s}\nonumber\\
=&\frac{1}{2}\sum_{k=1}^{r+s+1}\frac{(-1)^k(4^k-1)\zeta(2k+1)(2k)!}{(2\pi)^{2k}}
\sum_{i'=\max(2,2k-2r)-2}^{2s+1}(-1)^{i'+1}\binom{2s+1}{i'}\binom{2r+i'+1}{2k}\nonumber\\
&+\frac{1}{2}\sum_{k=1}^{r+s+1}\frac{(2k)!(-1)^k\zeta(2k+1)}{\pi^{2k}}
\binom{2s+1}{2k-2r-1}\nonumber\\
=&\frac{1}{2}\sum_{k=1}^{r}\frac{(-1)^{k+1}(4^k-1)\zeta(2k+1)(2k)!}{(2\pi)^{2k}}
\sum_{i'=0}^{2s+1}(-1)^{i'}\binom{2s+1}{i'}\binom{2r+i'+1}{2k}\nonumber\\
&+\frac{1}{2}\sum_{k=r+1}^{r+s+1}\frac{(-1)^{k+1}(4^k-1)\zeta(2k+1)(2k)!}{(2\pi)^{2k}}
\sum_{i'=2k-2r-2}^{2s+1}(-1)^{i'}\binom{2s+1}{i'}\binom{2r+i'+1}{2k}\nonumber\\
&-\frac{1}{2}\sum_{k=1}^{r+s+1}\frac{(2k)!(-1)^{k+1}\zeta(2k+1)}{\pi^{2k}}
\binom{2s+1}{2k-2r-1}.
\end{align}
But applying Lemma \ref{lemma2.2} gives us that
\begin{align}\label{eq4.5}
\sum_{i'=0}^{2s+1}(-1)^{i'}\binom{2s+1}{i'}\binom{2r+i'+1}{2k}
=-\binom{2r+1}{2k-2s-1}
\end{align}
and
\begin{align}\label{eq4.6}
\sum_{i'=2k-2r-1}^{2s+1} (-1)^{i'} \binom{2s+1}{i'} \binom{2r+i'+1}{2k}
=-\binom{2r+1}{2k-2s-1}.
\end{align}
Hence it follows from putting \eqref{eq4.5} and \eqref{eq4.6} into \eqref{eq4.4} that
\begin{align}\label{eq4.7}
D_{r,s}=&-\frac{1}{2}\sum_{k=1}^{r}\frac{(-1)^{k+1}(4^k-1)
\zeta(2k+1)(2k)!}{(2\pi)^{2k}}\binom{2r+1}{2k-2s-1}\nonumber\\
&-\frac{1}{2}\sum_{k=r+1}^{r+s+1}\frac{(-1)^{k+1}(4^k-1)
\zeta(2k+1)(2k)!}{(2\pi)^{2k}}\binom{2r+1}{2k-2s-1}\nonumber\\
&-\frac{1}{2}\sum_{k=1}^{r+s+1}\frac{(2k)!(-1)^{k+1}\zeta(2k+1)}{\pi^{2k}}
\binom{2s+1}{2k-2r-1}\nonumber\\
=&-\frac{1}{2}\sum_{k=1}^{r+s+1}\frac{(-1)^{k+1}(4^k-1)
\zeta(2k+1)(2k)!}{(2\pi)^{2k}}\binom{2r+1}{2k-2s-1}\nonumber\\
&-\frac{1}{2}\sum_{k=1}^{r+s+1}\frac{(2k)!(-1)^{k+1}\zeta(2k+1)}{\pi^{2k}}
\binom{2s+1}{2k-2r-1}.
\end{align}
Now substituting \eqref{eq4.7} into \eqref{eq4.3}, we obtain that
\begin{align}\label{eq4.8}
C_{r,s}=&\frac{1}{(2s+1)!}D_{r,s}\nonumber\\
=&-\frac{1}{2(2s+1)!}\Big(\sum_{k=1}^{r+s+1}\frac{(-1)^{k+1}(4^k-1)
\zeta(2k+1)(2k)!}{(2\pi)^{2k}}\binom{2r+1}{2k-2s-1}\nonumber\\
&-\sum_{k=1}^{r+s+1}\frac{(2k)!(-1)^{k+1}\zeta(2k+1)}{\pi^{2k}}
\binom{2s+1}{2k-2r-1}\Big).
\end{align}
Then putting \eqref{eq4.8} into \eqref{eq4.2}, one gets that
\begin{align}\label{eq4.9}
&T(r,s)\nonumber\\
=&\frac{-2}{(2r+1)!}\Big(\frac{\pi}{2}\Big)^{2r+2s+2}C_{r,s}\nonumber\\
=&\frac{1}{(2r+1)!}\Big(\frac{\pi}{2}\Big)^{2r+2s+2}
\frac{1}{(2s+1)!}\Big(\sum_{k=1}^{r+s+1}\frac{(-1)^{k+1}(4^k-1)
\zeta(2k+1)(2k)!}{(2\pi)^{2k}}\binom{2r+1}{2k-2s-1}\nonumber\\
&+\sum_{k=1}^{r+s+1}\frac{(2k)!(-1)^{k+1}\zeta(2k+1)}{\pi^{2k}}
\binom{2s+1}{2k-2r-1}\Big)\nonumber\\
=&\frac{1}{(2r+1)!}\Big(\frac{\pi}{2}\Big)^{2r+2s+2}
\frac{1}{(2s+1)!}\sum_{k=1}^{r+s+1}\frac{(-1)^{k+1}
\zeta(2k+1)(2k)!}{\pi^{2k}}\Big(1-\frac{1}{4^k}\Big)\nonumber\\
&\times\frac{(2r+1)!}{(2k-2s-1)!(2r+2s+2-2k)!}\nonumber\\
&+\frac{1}{(2r+1)!}\Big(\frac{\pi}{2}\Big)^{2r+2s+2}
\frac{1}{(2s+1)!}\sum_{k=1}^{r+s+1}\frac{(2k)!(-1)^{k+1}\zeta(2k+1)}{\pi^{2k}}\nonumber\\
&\times\frac{(2s+1)!}{(2k-2r-1)!(2r+2s+2-2k)!}\nonumber\\
=&\sum_{k=1}^{r+s+1}(-1)^{k+1}
\zeta(2k+1)\Big(1-\frac{1}{4^k}\Big)
\frac{(2k)!}{(2s+1)!(2k-2s-1)!}
\frac{\pi^{2r+2s+2-2k}}{2^{2r+2s+2}(2r+2s+2-2k)!}\nonumber\\
&+\sum_{k=1}^{r+s+1}(-1)^{k+1}\zeta(2k+1)
\frac{(2k)!}{(2r+1)!(2k-2r-1)!}
\frac{\pi^{2r+2s+2-2k}}{2^{2r+2s+2}(2r+2s+2-2k)!}\nonumber\\
=&\sum_{k=1}^{r+s+1}(-1)^{k+1}\zeta(2k+1)\Big(1-\frac{1}{4^k}\Big)
\binom{2k}{2s+1}\frac{\pi^{2r+2s+2-2k}}{2^{2r+2s+2-2k}(2r+2s+2-2k)!}\nonumber\\
&+\sum_{k=1}^{r+s+1}(-1)^{k+1}\zeta(2k+1)\binom{2k}{2r+1}
\frac{\pi^{2r+2s+2-2k}}{2^{2r+2s+2}(2r+2s+2-2k)!}\nonumber\\
=&\sum_{k=1}^{r+s+1}(-1)^{k+1}\Big[\binom{2k}{2r+1}
+\Big(1-\frac{1}{2^{2k}}\Big)\binom{2k}{2s+1}\Big]
\frac{\pi^{2r+2s+2-2k}}{2^{2r+2s+2}(2r+2s+2-2k)!}\zeta(2k+1).
\end{align}

Finally, one has by \eqref{eq1.2} that
\begin{align}\label{eq4.10}
& \frac{\pi^{2r+2s+2-2k}}{2^{2r+2s+2}(2r+2s+2-2k)!}\nonumber\\
=&\frac{1}{2^{2k}}\frac{\pi^{2r+2s+2-2k}}{2^{2r+2s+2-2k}(2r+2s+2-2k)!}\nonumber\\
=&\frac{1}{2^{2k}}t(\mathop{\underbrace{2,2,\ldots,2}}_{r+s+1-k}).
\end{align}
Hence putting \eqref{eq4.10} into \eqref{eq4.9} yields that
\begin{align*}
& T(r,s)\nonumber\\
=&\sum_{k=1}^{r+s+1}(-1)^{k+1}\Big[\binom{2k}{2r+1}
+\Big(1-\frac{1}{2^{2k}}\Big)\binom{2k}{2s+1}\Big]\frac{1}{2^{2k}}
t(\mathop{\underbrace{2,2,\ldots,2}}_{r+s+1-k})\zeta(2k+1)\\
=& \sum_{k=1}^{r+s+1}(-1)^{k-1}d^{(k)}_{r,s}\frac{1}{2^{2k}}\zeta(2k+1)
t(\mathop{\underbrace{2,2,\ldots,2}}_{r+s+1-k})
\end{align*}
as one wants. 

This completes the proof of Theorem \ref{theorem1.2}.
\hfill$\Box$\\

 \end{document}